\documentclass[a4paper, 12pt, twoside,openright]{article}

\usepackage{etex} 

\usepackage{enumerate}
\usepackage{enumitem}

\usepackage[section]{placeins}	
\usepackage{epigraph}
\usepackage[font={small}]{caption}
\usepackage{appendix}

\usepackage{amsmath,amssymb,mathrsfs}
\usepackage{amsfonts}			
\usepackage{nicefrac}			

\usepackage{bbm} 

\usepackage{stmaryrd}
\usepackage{calc,ifthen,xspace}

\usepackage[all,cmtip]{xy}
\usepackage{array}
\usepackage{multirow}
\usepackage{booktabs}
\usepackage{colortbl}
\usepackage{tabularx}
\usepackage{multirow}
\usepackage{threeparttable}

\newcolumntype{C}{>$c<$}

\usepackage{graphicx}			
\usepackage{subcaption}
\usepackage{pdfpages}
\usepackage{rotating}
\usepackage{pgfplots}
	\usepgfplotslibrary{groupplots}
	
\usepackage{tikz}
	\usetikzlibrary{backgrounds,automata}
	\pgfplotsset{width=7cm,compat=1.3}
	\pgfplotsset{every linear axis/.append style={
		/pgf/number format/.cd,
		use comma,
		1000 sep={\,},
	}}
\usepackage{eso-pic}
\usepackage{import}

\usepackage{tikz-cd}
\usetikzlibrary{knots}

\usepackage{xspace}
\usepackage{textcomp}
\usepackage{array}
\usepackage{hyphenat}

\usepackage[pdftex,pdfborder={0 0 0},
			colorlinks=true,
			linkcolor=blue,
			citecolor=red,
			pagebackref=true,
			]{hyperref}	

\usepackage[top=2.5cm, bottom=2cm, left=3cm, right=2.5cm,
			headheight=15pt]{geometry}

\AtBeginDocument{}

\renewcommand{\leq}{\ensuremath{\leqslant}} 
\renewcommand{\geq}{\ensuremath{\geqslant}}

\DeclareMathOperator{\Z}{\mathbb Z}

\DeclareMathOperator{\Lie}{\mathcal L}
\DeclareMathOperator{\Der}{Der}
\DeclareMathOperator{\ad}{ad}

\DeclareMathOperator{\Hom}{Hom}
\DeclareMathOperator{\Aut}{Aut}

\DeclareMathOperator{\Act}{Act}

\usepackage{amsthm}

\makeatletter
\newtheorem*{rep@theorem}{\rep@title}
\newcommand{\newreptheorem}[2]{%
\newenvironment{rep#1}[1]{%
 \def\rep@title{#2 \ref{##1}}%
 \begin{rep@theorem}}%
 {\end{rep@theorem}}}
\makeatother

\theoremstyle{plain}
\newtheorem{theo}{Theorem}
\newreptheorem{theo}{Theorem}
\newtheorem{fait}[theo]{Fact} 
\newtheorem{lem}[theo]{Lemma}
\newtheorem{prop}[theo]{Proposition}
\newreptheorem{prop}{Proposition}
\newtheorem{propdef}[theo]{Proposition-definition}
\newtheorem{cor}[theo]{Corollary}
\newreptheorem{cor}{Corollary}

\newcounter{PB}
\newtheorem{pb}[PB]{Problem}
\newreptheorem{pb}{Problem}

\numberwithin{equation}{theo}

\theoremstyle{definition}
\newtheorem{defi}[theo]{Definition}

\newtheorem{ex}[theo]{Example}
\newtheorem{rmq}[theo]{Remark}

\renewcommand{\k}{\Bbbk}

\newcommand{\Enstq}[2]{\left\{\ #1\ \middle|\ #2\ \right\}}    
\newcommand{\Id}[1][]{\mathbbm 1_#1}                             
\newcommand{\cpx}[1]{#1_*}

\let\oldpagenumbering\pagenumbering
\renewcommand{\pagenumbering}[1]{%
	\cleardoublepage
	\oldpagenumbering{#1}
}

\author{Jacques \scshape{Darn\'e}}
\title{On the Andreadakis problem for subgroups of $IA_n$}
\date{August 20, 2018}

\hypersetup{
    pdfauthor={Jacques Darn\'e},
    pdfsubject={Preprint},
    pdftitle={On the Andreadakis problem for subgroups of IAn},
    pdfkeywords={algebraic topology, group theory, automorphisms of free groups}
}

\begin{document}

\maketitle

\begin{abstract}
Let $F_n$ be the free group on $n$ generators. Consider the group $IA_n$ of automorphisms of $F_n$ acting trivially on its abelianization. There are two canonical filtrations on $IA_n$: the first one is its lower central series $\Gamma_*$; the second one is the Andreadakis filtration $\mathcal A_*$, defined from the action on $F_n$. The Andreadakis problem consists in understanding the difference between these filtrations. Here, we show that they coincide when restricted to the subgroup of triangular automorphisms, and to the pure braid group.
\end{abstract}

\section*{Introduction}
\addcontentsline{toc}{section}{Introduction}

The group $\Aut(F_n)$ of automorphisms of a free group has a very rich structure, which is somewhat ill-understood. It is linked to various groups appearing in low-dimensional topology : it contains the mapping class group, the braid group, the loop braid group, etc. By looking at its action on $F_n^{ab} \cong \Z^n$, we can decompose $\Aut(F_n)$ as an extension of $GL_n(\Z)$ by $IA_n$, the latter being the subgroup of automorphisms acting trivially on $\Z^n$. By analogy with the case of the mapping class group, $IA_n$ is known as the \emph{Torelli subgroup} of $\Aut(F_n)$. An explicit finite set of generators of $IA_n$ has been known for a long time \cite{Nielsen} -- see also \cite[5.6]{BBM} and our appendix.
Nevertheless, the structure of $IA_n$ remains largely mysterious. For instance, $IA_3$ is not finitely presented \cite{Krstic}, and it is not known whether $IA_n$ is finitely presented for $n > 3$.

One of the most prominent questions concerning the structure of this Torelli group is the Andreadakis problem. Consider the lower central series $F_n = \Gamma_1(F_n) \supseteq \Gamma_2(F_n) \supseteq\cdots$. From it we can define the \emph{Andreadakis filtration} $IA_n = \mathcal A_1 \supseteq \mathcal A_2 \supseteq \cdots $:  we define $\mathcal A_j$ as the subgroup of automorphisms acting trivially on $F_n/\Gamma_{j+1}(F_n)$ (which is the free nilpotent group of class $j$ on $n$ generators). This filtration of $IA_n$ is central (it is even \emph{strongly central}), so it contains the minimal central filtration on $IA_n$, its lower central series: 
\[\forall k \geq 1,\ \mathcal A_k \supseteq \Gamma_k(IA_n).\] 
\begin{pb}[Andreadakis]\label{pb_Andreadakis}
What is the difference between $\mathcal A_*$ and $\Gamma_*(IA_n)$? 
\end{pb}

Andreadakis conjectured that the filtrations were the same \cite[p.\ 253]{Andreadakis}. In \cite{Bartholdi}, Bartholdi disproved the conjecture, using computer calculations.  He then tried to prove that the two filtrations were the same up to finite index, but in the erratum \cite{Bartholdi1}, he showed that even this weaker statement cannot be true. His latter proof uses the $L$-presentation of $IA_n$ given in \cite{Day-Putman}, to which he applies algorithmic methods described in \cite{BBH} to calculate (using the software GAP) the first degrees of the graded groups associated to each filtration. 

\medskip

The present paper is devoted to the study of the Andreadakis problem when restricted to some subgroups of the Torelli group $IA_n$. Precisely, if $G$ is a subgroup of $IA_n$, we can consider the two filtrations induced on $G$ by our original filtrations, and we can compare them to the lower central series of $G$:
\begin{equation}\label{incl_on_subgroups}
\Gamma_* (G)\ \subseteq\ \Gamma_*(IA_n)\cap G\ \subseteq\ \mathcal A_* \cap G.
\end{equation}
\begin{pb}[Andreadakis problem for subgroups of $IA_n$]\label{Andreadakis_on_subgroups}
For which subgroups $G$ of $IA_n$ are the above inclusions equalities ?
\end{pb} 

\begin{defi}
We say that the \emph{Andreadakis equality} holds for a subgroup $G$ of $IA_n$ when $\Gamma_*(G) = \mathcal A_* \cap G$.
\end{defi}

Obviously, this is not always the case: for instance, the Andreadakis equality does not hold for the cyclic group generated by an element of $\Gamma_2(IA_n)$. However, for some nicely embedded groups, we can hope that it could hold, and it is indeed the case:
\begin{theo}[Th.\ \ref{Andreadakis_for_triangulars}, Cor.\ \ref{Andreadakis_for_triangular_McCool} and Th.\ \ref{Andreadakis_for_braids}]\label{IAn+ et Pn} 
Let $G$ be the subgroup of triangular automorphisms $IA_n^+$, the triangular McCool subgroup $P\Sigma_n^+,$ or the pure braid group $P_n$ acting \emph{via} the Artin action. Then the Andreadakis equality holds for $G$:
\[\cpx\Gamma (G) = \cpx \Gamma(IA_n)\cap G = \mathcal A_* \cap G.\]
\end{theo}
The statement about triangular automorphisms has independently been obtained by T. Satoh  \cite{Satoh_triangulaire}.

The subgroup $IA_n^+$ is introduced in Definition \ref{def_triangulaires}. We treat the cases of $IA_n^+$ and $P_n$ in Sections \ref{Section_triangular} and \ref{Section_braids} ; the proof in the case of $P\Sigma_n^+$ is a straightforward adaptation of the proof for $IA_n^+$.  The methods used in both cases are very similar: both use a decomposition as an iterated almost-direct product, and fit in the general framework we introduce in Section \ref{Section_dec}. Sections \ref{Generalities} to \ref{section_exactness} consist mainly of reminders from \cite{Darné}, with some additional material, especially in Paragraph \ref{par_def_actions}, where we present a general adjonction involving semi-direct products, and in Paragraph \ref{par_LCS_of_sdp}, where we write down a description of the lower central series of a semi-direct product of groups.

\bigskip

\noindent
\textbf{Acknowledgements}: This work is part of the author's PhD thesis. The author is indebted to his advisors, Antoine Touz\'e and Aur\'elien Djament, for their constant support, countless helpful discussions, and numerous remarks and comments on earlier versions of the present paper. He also thanks Gwena\"el Massuyeau and Takao Satoh, who kindly agreed to be the reviewers of his thesis, for their useful observations and comments.

\tableofcontents

\numberwithin{theo}{section}
\section{Strongly central filtrations}\label{Generalities}

Throughout the paper, $G$ will denote an arbitrary group,. The left and right action of $G$ on itself by conjugation are denoted respectively by $x^y = y^{-1}xy$ and ${}^y\! x = yxy^{-1}$.
The \emph{commutator} of two elements $x$ and $y$ in $G$ is $[x,y]:= xyx^{-1}y^{-1}$.
If $A$ and $B$ are subsets of $G$, we denote by $[A, B]$ the subgroup generated by the commutators $[a,b]$ with $(a,b) \in A \times B$. 
We denote the \emph{abelianization} of $G$ by $G^{ab}:= G/[G,G]$ and its lower central series by $\Gamma_*(G)$, that is : 
\[G =: \Gamma_1(G) \supseteq [G,G] =: \Gamma_2(G) \supseteq [G,\Gamma_2(G)] =: \Gamma_3(G)\supseteq \cdots\]

We recall the definition of the category $\mathcal{SCF}$ introduced in \cite{Darné}.
\begin{defi}
A strongly central filtration $G_*$ is a nested sequence of groups  $G_1 \supseteq G_2 \supseteq G_3 \cdots$ such that $[G_i, G_j] \subseteq G_{i+j}$ for all $i, j \geq 1$. These filtrations are the objects of a category $\mathcal{SCF}$, where morphisms from $G_*$ to $H_*$ are those group morphisms from $G_1$ to $H_1$ sending each $G_i$ into $H_i$.
\end{defi}

Recall that this category has the following features \cite{Darné}:
\begin{itemize}
\item  There are forgetful functors $\omega_i: G_* \mapsto G_i$ from $\mathcal{SCF}$ to the category of groups. Since the lower central series $\Gamma_*(G)$ is the minimal strongly central series on a group $G$, the functor $\Gamma_*$ is left adjoint to $\omega_1$.

\item There is a functor from the category $\mathcal{SCF}$ to the category $\mathcal Lie_{\Z}$ of Lie rings (\emph{i.e.\ }Lie algebras over $\Z$), given by:
$G_* \mapsto \Lie(G_*):= \bigoplus{G_i/G_{i+1}},$ 
where Lie brackets are induced by group commutators.

\item  The category $\mathcal{SCF}$ is complete and cocomplete. To compute limits (resp. colimit), just endow the corresponding colimit of groups with the maximal (resp. minimal) compatible filtration.
\item  It is homological \cite[def. 4.1.1]{BB}. This means essentially that the usual lemmas of homological algebra (the nine lemma, the five lemma, the snake lemma, etc.) are true there. 
\item  It is action-representative (see Paragraph \ref{actions_in_SCF} below). 
\end{itemize}

In a homological category, we need to distinguish between usual epimorphisms (resp. monomorphisms) and \emph{regular} ones, that is, the ones obtained as coequalizers (resp. equalizers). In $\mathcal{SCF}$, the former are the $u$ such that $u_1 = \omega_1(u)$ is an epimorphism (resp. a monomorphism), whereas the latter are \emph{surjections} (resp. \emph{injections}):
\begin{defi}\label{def_inj-surj}
Let $u: G_* \longrightarrow H_*$ be a morphism in $\mathcal{SCF}$. It is called an \emph{injection} (resp.\ a \emph{surjection}) when $u_1$ is injective (resp.\ surjective) and $u^{-1}(H_i) = G_i$ (resp.\ $u(G_i) = H_i$) for all $i$.
\end{defi}

We can use this to give an explicit interpretation of the general notion of \emph{short exact sequences} \cite[Def 4.1.5]{BB} in $\mathcal{SCF}$: 
$\begin{tikzcd}[column sep=small]
                  1 \ar[r] &G_* \ar[r, "u"] & H_* \ar[r, "v"] & K_* \ar[r] &1
                  \end{tikzcd}$ 
is a short exact sequence if and only if:
\begin{equation}\label{SEC_in_SCF}
\left\{\begin{array}{l}
               u \text{ is an injection}, \\
               v \text{ is a surjection},\\
               u(G_1) = \ker(v).
         \end{array}
\right.
\end{equation}

\section{Semi-direct products and actions}

An action of group on another one by automorphisms, or of a Lie algebra on another one by derivations, are particular cases of a general notion of actions in a homological (or even only protomodular) category. We recall here the definitions and elementary properties of such actions, and what actions are in $\mathcal{SCF}$.

\subsection{Actions: an abstract definition}\label{par_def_actions}

In this paragraph, we use the language of \cite{BB}. A concise and accessible reference on the subject is \cite{Hartl}. The reader not familiar with this language can replace the category $\mathcal C$ by his favorite example (groups or Lie algebras, for instance).

\begin{defi}\label{def_action}
Let $\mathcal C$ be a protomodular category. If $X$ and $Z$ are two objects of $\mathcal C$, we define an \emph{action} of $Z$ on $X$ as a split extension (with a given splitting):
\[\begin{tikzcd} X \ar[hook, r] & Y \ar[two heads, r] & Z. \ar[l, bend right]
\end{tikzcd}\]
When such an action is given, we will say that $Z$ \emph{acts on} $X$, and write: $Z \circlearrowright X$.
Actions in $\mathcal C$ form a category $\mathcal Act(\mathcal C)$, morphisms between two actions being the obvious ones.
\end{defi}

\begin{ex}
The category $\mathcal Act(\mathcal Grp)$ is the usual category of group actions on one another by automorphisms: a morphism from $G \circlearrowright H$ to $G' \circlearrowright H'$ is given by a morphism $u: G \rightarrow G'$ and a morphism $v: H \rightarrow H'$, $v$ being $u$-equivariant, that is:
$\forall g \in G,\ \forall h \in H,\ v(g \cdot h) = u(g) \cdot v(h).$
\end{ex}

\begin{defi}\label{adjonction_ad-psd}
We denote by $\ltimes$ the functor from $\mathcal Act(\mathcal C)$ to $\mathcal C$ sending an action 
\[\begin{tikzcd} X \ar[r, hook] &Y \ar[r, two heads] &Z \ar[l, bend right]\end{tikzcd}\] 
on $Y$ (which is called $Z \ltimes X$), and by $\ad$ the functor from $\mathcal C$ to $\mathcal Act(\mathcal C)$ sending an object $C$ on \emph{the adjoint action}:
\[\begin{tikzcd}[column sep = huge, ampersand replacement=\&] C \ar[r, hook, "\iota_1 = \scalebox{0.7}{$\begin{pmatrix} 1 \\ 0 \end{pmatrix}$}"] \& C^2 \ar[r, two heads, swap, "\pi_2 = \scalebox{0.7}{$\begin{pmatrix} 0 & 1 \end{pmatrix}$}"] \& C. \ar[l, swap, bend right, "\delta = \scalebox{0.7}{$\begin{pmatrix} 1 \\ 1 \end{pmatrix}$}"]
\end{tikzcd}\]
\end{defi}

\begin{prop}
The above functors are adjoint to each other:
\[\ltimes:
\begin{tikzcd}
\mathcal Act(\mathcal C) \ar[r, shift left] &\mathcal C \ar[l, shift left]
\end{tikzcd}
: \ad.\]
\end{prop}

\begin{proof}
Let $C$, $X$ and $Z$ be objects of $\mathcal C$. Let an action of $Z$ on $X$ be given. A morphism $\varphi: Z \ltimes X \rightarrow C$ induces a morphism between actions:
\[\begin{tikzcd} 
X \ar[r, hook, "i"] \ar[d, swap, "\varphi i"] &Z \ltimes X \ar[r, two heads, swap, "p"] \ar[d, swap, "{(\varphi, \varphi s p)}"]  &Z \ar[l, bend right, swap, "s"] \ar[d, "\varphi s"] \\
C \ar[r, hook, "\iota_1"] & C^2 \ar[r, two heads, swap, "\pi_2"] &C. \ar[l, bend right,  swap, "\delta"]
\end{tikzcd}\]
Conversely, a morphism between these actions gives in particular a morphism $\varphi: Z \ltimes X \rightarrow C^2 \overset{\pi_1}{\rightarrow} C$. One can easily check that these constructions are inverse to each other.
\end{proof}

\begin{ex}
When $\mathcal C = Grp$, the functor $\ad$ sends $G$ to the conjugation action $G \circlearrowright G$. The semi-direct product can be \emph{defined} as this functor's left adjoint.
\end{ex}

Any functor $F: \mathcal C \rightarrow \mathcal D$ does preserve split epimorphisms. As a consequence:
\begin{fait}\label{foncteur_et_psd}
Let $F: \mathcal C \rightarrow \mathcal D$ be a functor between protomodular categories. Then, for any action $Z \circlearrowright X$ in $\mathcal C$,
$F(Z \ltimes X) = F(Z) \ltimes \ker(F(Z \ltimes X \twoheadrightarrow Z)).$
\end{fait}
This allows us to define an induced functor $F_\#$ between the categories of actions:
\begin{equation}\label{def_Act(F)}
F_\#\left(
\begin{tikzcd} 
X \ar[r, hook, "i"] &Y \ar[r, two heads, swap, "p"]  &Z \ar[l, bend right, swap, "s"] 
\end{tikzcd}
\right):= 
\begin{tikzcd} 
\ker(Fp) \ar[r, hook] &F(Y) \ar[r, two heads, swap, "Fp"]  &F(Z). \ar[l, bend right, swap, "Fs"] 
\end{tikzcd}
\end{equation}
Remark that the description of $F_\#$ is particularly simple when $F$ preserves kernels of split epimorphisms: then $\ker(Fp) = F(X)$. 

\begin{rmq}
This construction $F \mapsto F_\#$ makes the construction of the category of actions into a functor from protomodular categories to categories. As one sees easily, it is a $2$-functor and, as such, it preserves adjunctions. Moreover, if 
$F: \begin{tikzcd}
\mathcal C \ar[r, shift left] &\mathcal D \ar[l, shift left] 
\end{tikzcd}: G$ 
is an adjunction, then we can write the following diagram:
\begin{equation}\label{diag_adj_actions}
\begin{tikzcd}
\mathcal Act(\mathcal C) \ar[r, shift left, "\ltimes"] \ar[d, shift right, swap, "F_\#"] 
&\mathcal C \ar[l, shift left, "\ad"] \ar[d, shift right, swap, "F"]  \\
\mathcal Act(\mathcal D) \ar[r, shift left, "\ltimes"] \ar[u, shift right, swap, "G_\#"] 
&\mathcal D. \ar[l, shift left, "\ad"] \ar[u, shift right, swap, "G"]
\end{tikzcd}
\end{equation}
Since $G$ commutes to limits, the square of right adjoints is commutative: $G_\# \circ \ad = \ad \circ G$. Obviously, the square of left adjoints also commute (which is an equivalent statement).
\end{rmq}

\subsection{Representability of actions}

The set $\Act(Z,X)$ of actions of $Z$ on $X$ is a contravariant functor in $Z$: the restriction of an action along a morphism is defined \emph{via} a pullback. In $\mathcal Grp$, as in $\mathcal Lie$, this functor is representable, for any $X$. Indeed, an action of a group $K$ on a group $G$ is given by a morphism $K \longrightarrow \Aut(G)$. Similarly, an action of a Lie algebra $\mathfrak k$ on a Lie algebra $\mathfrak g$ is given by a morphism $\mathfrak k \longrightarrow \Der(\mathfrak g)$.
The situation when actions are representable has notably been studied in \cite{BJK}. The following terminology was introduced in \cite[Def. 1.1]{Borceux-Bourn}:
\begin{defi}
A protomodular category $\mathcal C$ is said to be \emph{action-representative} when the functor $\Act(-,X)$ is representable, for any object $X \in \mathcal C$.
\end{defi}
A representative for $\Act(-,X)$ is a universal action on $X$. Explicitly, it is an action of an object $\mathcal A(X)$ on $X$ such that any action $Z \circlearrowright X$ is obtained by restriction along a unique morphism $Z \rightarrow \mathcal A(X)$.

\subsection{Actions in \texorpdfstring{$\mathcal{SCF}$}{SFC}}
\label{actions_in_SCF}

Using the explicit description of exact sequences in $\mathcal{SCF}$ given at the end of the first section \eqref{SEC_in_SCF}, we can describe explicitly what an action in $\mathcal{SCF}$ is:
\begin{prop}\textup{\cite[Prop. 1.20]{Darné}}.\label{actions_dans_SCF}
An action $K_* \circlearrowright G_*$ in $\mathcal{SCF}$ is the data of a group action of $K = K_1$ on $G = G_1$ satisfying:
\[\forall i, j \geq 1,\ [K_i, G_j] \subseteq G_{i+j}.\]
\end{prop}

Using this description, one can show that actions are representable in $\mathcal{SCF}$:

\begin{theo}\label{Kaloujnine} \textup{\cite[Th. 1.16, Prop. 1.22]{Darné}}.
Let $G_*$ be a strongly central series. Let $j \geq 1$ be an integer. Define $\mathcal A_j(G_*) \subseteq \Aut(G_*)$ to be :
\begin{equation}\label{def_A_*}
\mathcal A_j(G_*) = \Enstq{\sigma \in \Aut(G_*)}{\forall i \geq 1,\ [\sigma, G_i] \subseteq G_{i+j}},
\end{equation} 
where the commutator is computed in $G_1 \rtimes \Aut(G_*)$, that is: $[\sigma, g] = \sigma(g)g^{-1}$. That is, $\mathcal A_j(G_*)$ is the group of automorphisms of $G_*$ acting trivially on every quotient $G_i/G_{i+j}$.
Then $\mathcal A_*(G_*)$ is a strongly central series. Moreover, it acts canonically  on $G_*$, and this action is universal. In particular, the category $\mathcal{SCF}$ is action-representative.
\end{theo}

If a group $K$ acts on a group $G$, and $G_*$ is a strongly central filtration on $G = G_1$, we can pull back the canonical filtration $\mathcal A_*(G_*)$ by the associated morphism:
\[K \longrightarrow \Aut(G).\]
This gives a strongly central filtration $\mathcal A_*(K,G_*)$, maximal amongst strongly central filtrations on subgroups of $K$ which act on $G_*$ \emph{via} the given action $K \circlearrowright G$. It can be described explicitly as:
\begin{equation}\label{A_on_K}
\mathcal A_j(K,G_*) = \Enstq{k \in K}{\forall i \geq 1,\ [k, G_i] \subseteq G_{i+j}}\ \subseteq K.
\end{equation}

\subsection{The Andreadakis problem}

Let $G$ be a group. We denote $\Lie(\Gamma_*(G))$ by $\Lie(G)$, and we call it \emph{the Lie ring of $G$}. As products of commutators become sums of brackets inside the Lie algebra, the following fundamental property follows from the definition of the lower central series:
\begin{prop}\label{engdeg1} \textup{\cite[Prop. 1.19]{Darné}}.
The Lie ring $\Lie(G)$ is \emph{generated in degree $1$}. Precisely, it is generated (as a Lie ring) by $\Lie_1(G) = G^{ab}$.
\end{prop}

Consider $\mathcal A_*(G) := \mathcal A_*(\Gamma_*G)$. The group $\mathcal A_1(G)$ is the group of automorphisms of $G$ acting trivially on $\Lie(G)$. But this Lie algebra is generated in degree $1$. As a consequence, $\mathcal A_1(G)$ is the subgroup of automorphisms acting trivially on the abelianization $G^{ab} = \Lie_1(G)$, denoted by $IA_G$. The filtration $\mathcal A_*(G)$, being strongly central on $\mathcal A_1(G) = IA_G$, contains $\Gamma_*(IA_G)$. We are thus led to the problem of comparing these filtrations ; this is the \emph{Andreadakis problem} (Problem \ref{pb_Andreadakis}), which is a crucial question when trying to understand the structure of automorphism groups of residually nilpotent groups, in particular when trying to understand the structure of $\Aut(F_n)$.

\begin{reppb}{pb_Andreadakis}[Andreadakis]
How close is the inclusion $\Gamma_*(IA_G) \subseteq \mathcal A_*(G)$ to be an equality  ? 
\end{reppb}

\section{Exactness of the Lie functor}\label{section_exactness}

In this section, we investigate the Lie algebra of a semi-direct product of groups, and we recall the construction of the Johnson morphism associated with an action in $\mathcal{SCF}$. Both rely on the following fundamental proposition:
\begin{prop}\label{exactness_of_L} \textup{\cite[Prop. 1.24]{Darné}}.
The Lie functor $\Lie: \mathcal{SCF} \longrightarrow \mathcal Lie_{\Z}$ is \emph{exact}, \emph{i.e.\ }it preserves short exact sequences.
\end{prop}

\begin{cor}
The functor $\Lie$ preserves actions. In other words, if $K_* \circlearrowright H_*$ is an action in $\mathcal{SCF}$, then $\Lie(H_* \rtimes K_*) = \Lie(H_*) \rtimes \Lie(K_*).$
\end{cor}

\subsection{Lower central series of a semi-direct product of groups}\label{par_LCS_of_sdp}

We can use the tools introduced so far to study the lower central series of a semi-direct product of groups, and its Lie algebra.
Precisely, let $G = H \rtimes K$ be a semi-direct product of groups. The functor $F = \Gamma$ preserves split epimorphisms, whence a decomposition into a semi-direct product of strongly central series:
\[\Gamma_* G = H_* \rtimes \Gamma_* K,\]
where $H_i$ is the kernel of the split projection:
\[\begin{tikzcd}
  \Gamma_i G   \ar[r, two heads] &
  \Gamma_i K.  \ar[l, bend right, dashed]
  \end{tikzcd}\]
The aim of the present paragraph is to give an explicit description of $H_*$ (that is, using the notations of \ref{par_def_actions}, to describe $\Gamma_\#(K \circlearrowright H)$) and to identify the conditions under which $H_*$ is equal to $\Gamma_*H$.

\medskip

Let us begin by introducing a general construction:

\begin{propdef}\label{SCD_relative}
Let $G$ be a group, and $H$ a normal subgroup. We define a strongly central filtration $\cpx\Gamma^G (H)$ on $H$ by:
\[\begin{cases} \Gamma_1^G (H):= H, \\ \Gamma_{k+1}^G:= [G, \Gamma_k^G (H)]. \end{cases}\]
\end{propdef}

\begin{proof}
The inclusions $\Gamma_{k+1}^G \subseteq \Gamma_k^G$ are obtained by induction on $k$, the first one being the normality of $H$ in $G$.
The strong centrality statement is obtained by induction, using the 3-subgroups lemma.
\end{proof}

Since $[G,\Gamma_k^G(H)]\subseteq \Gamma_k^G(H)$, the $\Gamma_k^G(H)$ are normal in $G$. In fact, we have: 
\[[G,\Gamma_k^G(H)]\subseteq \Gamma_{k+1}^G(H) \text{, so }G = \mathcal A_1(G,\Gamma_*^G(H)).\] 
As a consequence, $\mathcal A_*(G,\Gamma_*^G(H))$ is a strongly central filtration on the whole of $G$, so it contains $\Gamma_*(G)$. Thus $\Gamma_*(G)$ acts on $\Gamma_*^G(H)$. Moreover, it is clear that $\Gamma_*^G(H)$ is the minimal strongly central filtration on $H$ such that the action $G$ on $H$ induces an action of $\Gamma_*(G)$.

Now, let $K \circlearrowright H$ be a group action. We can apply the above construction to $G = H \rtimes K \supseteq H$. We will write $\Gamma_*^K(H)$ for $\Gamma_*^{H \rtimes K}(H)$ (this will not cause any confusion: if $H$ is a normal subgroup of a group $G$, then $\Gamma_*^G(H) = \Gamma_*^{H \rtimes G}(H)$, for the semi-direct product associated to the conjugation action of $G$ on $H$). Using these constructions, we can identify the filtration $H_*$ defined above:
\begin{prop}\label{Gamma(PSD)_0}
If a group $K$ acts on a group $H$, then:
\[\Gamma_*(H \rtimes K) = \Gamma_*^K(H) \rtimes \Gamma_*H.\]
\end{prop}

\begin{proof}
Since $\Gamma_*(K) \subseteq \Gamma_*(H \rtimes K)$ acts on $\Gamma_*^K(H)$, the filtration $\Gamma_*^K(H) \rtimes \Gamma_*(K)$ is strongly central on $H \rtimes K$, so it contains $\Gamma_*(H \rtimes K)$. The other inclusion follows directly from the definitions.
\end{proof}

\begin{rmq}
We have thus identified $\Gamma_\# = \mathcal Act(\Gamma)$ (see Paragraph \ref{par_def_actions}):
\[\Gamma_\#: K \circlearrowright H \longmapsto \Gamma_*(K) \circlearrowright \Gamma_*^K(H).\]
In this context, the diagram \eqref{diag_adj_actions} reads:
\[\begin{tikzcd}
  \mathcal Act(\mathcal{SCF}) \ar[r, shift left, "\ltimes"] \ar[d, shift right, swap, "\Gamma_\#"]
  & \mathcal{SCF}  \ar[l, shift left, "c"] \ar[d, shift right, swap, "\Gamma"] \\
  \mathcal Act(\mathcal Grp) \ar[r, shift left, "\ltimes"] \ar[u, shift right, swap, " \omega_1"] 
  & \mathcal Grp.  \ar[l, shift left, "c"] \ar[u, shift right, swap, " \omega_1"] \\
\end{tikzcd} 
\]
\end{rmq}

We now describe the conditions under which $\Gamma_*^K(H) = \Gamma_*(H)$:

\begin{propdef}\label{Gamma(PSD)}
Let $H \rtimes K$ be a semi-direct product of groups. Then the following assertions are equivalent:
\begin{enumerate}
\item The action of $K$ on $H^{ab}$ is trivial,
\item $[K, H] \subseteq [H,H] = \Gamma_2(H),$ \label{item:1}
\item $K \circlearrowright H \text{ induces an action }\Gamma_*K \circlearrowright \Gamma_*H,$ \label{item:2}
\item $\Gamma_*^K(H) = \Gamma_*H,$ \label{item:3}
\item $\forall i,\ \Gamma_i(H \rtimes K) = \Gamma_i H \rtimes \Gamma_i K,$ \label{item:4}
\item $\Lie(H \rtimes K) \cong \Lie(H) \rtimes \Lie(K),$ \label{item:5}
\item $(H \rtimes K)^{ab} \cong H^{ab} \times K^{ab}.$ \label{item:6}
\end{enumerate}
When these conditions are satisfied, we will say that the semi-direct product $H \rtimes K$ is an \emph{almost-direct} one.
\end{propdef}

\begin{proof}
The second statement means that $K$ acts trivally on $H^{ab} = \Lie_1(H)$, hence on $\Lie(H)$, as this algebra is generated in degree one $1$ (proposition \ref{engdeg1}). But this means exactly that:
\[\forall i,\ [K, \Gamma_i H] \subseteq \Gamma_{i+1}H,\]
which is equivalent to $K$ being equal to $\mathcal A_1(K,\Gamma_*H)$. Since $\mathcal A_*(K,\Gamma_*H)$ is strongly central, this in turn is equivalent to $\Gamma_*K \subseteq \mathcal A_*(K,\Gamma_*H)$, which is exactly \eqref{item:2}.

The filtration $\Gamma_*H$ is the minimal strongly central series on $H$, and $\Gamma_*^K(H)$ is the minimal one on which $\Gamma_*K$ acts (through the given action of $K$ on $H$). Hence the equivalence with \eqref{item:3}. The assertion \eqref{item:4} is clearly equivalent to \eqref{item:3} and, using the exactness of $\Lie$ (Proposition \eqref{exactness_of_L}), we see that it implies \eqref{item:5}. The remaining implications  $\eqref{item:5} \Rightarrow  \eqref{item:6}  \Rightarrow \eqref{item:1}$ are straightforward.
\end{proof}

\subsection{Johnson morphisms}

In this paragraph also, we recall some material from \cite{Darné}.

\medskip

As a consequence of Proposition \ref{exactness_of_L}, the functor $\Lie$ preserves actions. Precisely, from an action in $\mathcal{SCF}$:
\[\begin{tikzcd} G_* \ar[r, hook] & H_* \ar[r, two heads] & K_*, \ar[l, bend right]
\end{tikzcd}\]
we get an action in the category of graded Lie rings:
\[\begin{tikzcd} \Lie(G_*) \ar[r, hook] & \Lie(H_*) \ar[r, two heads] & \Lie(K_*). \ar[l, bend right]
\end{tikzcd}\]
Such an action is given by a morphism of graded Lie rings:
\begin{equation}\label{Johnson}
\tau: \Lie(K_*) \longrightarrow \Der_*(\Lie(G_*)).
\end{equation}
The target is the (graded) Lie algebra of graded derivations: a derivation is of degree $k$ when it raises degrees of homogeneous elements by $k$.
\begin{defi}
The morphism \eqref{Johnson} is called the \emph{Johnson morphism} associated to the given action $K_* \circlearrowright G_*$.
\end{defi}
We can give an explicit description of this morphism: for $k \in K$, the derivation associated to $\bar{k}$ is induced by $[\bar{k},-]$ inside $\Lie(G_* \rtimes K_*) = \Lie(G_*) \rtimes \Lie(K_*)$, so it is induced by $[k,-]$ inside $G_* \rtimes K_*$.

\begin{ex}\label{déf_tau}
The Johnson morphism associated to the universal action $\mathcal A_*(G_*)\circlearrowright G_*$ is the Lie morphism $\tau: \Lie(\mathcal A_*(G_*)) \longrightarrow \Der_*(\Lie(G_*))$
induced by $\sigma \mapsto (x \mapsto \sigma(x)x^{-1})$.
\end{ex}

The \emph{Johnson morphism} turns out to be a powerful tool in the study of the Andreadakis filtration, thanks to the following injectivity statement:
\begin{lem}\textup{\cite[Lem.\ 1.28]{Darné}}\label{tau_inj}
Let $K_* \circlearrowright G_*$ be an action in $\mathcal{SCF}$. The associated Johnson morphism $\tau: \Lie(K_*) \longrightarrow \Der_*(\Lie(G_*))$ is injective if and only if $K_* = \mathcal A_*(K_1, G_*)$.
\end{lem}

\begin{ex}\label{Johnson_Fn}
If $G $ is a free group, then $\Lie(F_n)$ is the free Lie algebra $\mathfrak LV$ on the $\Z$-module $V = G^{ab}$ \cite[th.\ 4.2]{Lazard}. 
It is also free with respect to derivations, which can be considered as sections of a suitable projection -- see for instance \cite{Reutenauer}. In particular:
\[\Der_k(\mathfrak LV) \cong \Hom_{\k}(V, \mathfrak L_k V) \cong V^*\otimes \mathfrak L_k V.\]
The Andreadakis filtration $\mathcal A_* = \mathcal A_*(F_n)$  is the universal one acting on $\Gamma_*(F_n)$. Thus the associated Johnson morphism is an embedding:
\begin{equation}
\tau: \Lie(\mathcal A_*) \hookrightarrow \Der(\mathfrak LV).
\end{equation}
\end{ex}

\section{Decomposition of an induced filtration}\label{Section_dec}

\subsection{General setting}

Let $G$ be a group endowed with a strongly central filtration $\mathcal A_*$. Let $H \rtimes K$ be a subgroup of $G$ decomposing as a semi-direct product. Then we get, on the one hand, a semi-direct product $(\mathcal A_* \cap H) \rtimes (\mathcal A_* \cap K)$ of strongly central series. On the other hand, $\mathcal A_* \cap (H \rtimes K)$ is a strongly central filtration on $H \rtimes K$ containing the previous one.

\begin{prop}\label{dec_of_induced_SCF}
In the above setting, the following assertions are equivalent:
\begin{enumerate}[label={(\roman*)}]
\item \label{item1} $\mathcal A_* \cap (H \rtimes K) = (\mathcal A_* \cap H) \rtimes (\mathcal A_* \cap K),$
\item \label{item2} Inside $\Lie(\mathcal A_*)$,\ \ $\Lie(\mathcal A_* \cap H) \cap \Lie(\mathcal A_* \cap K) = 0.$
\end{enumerate}
When they are satisfied, we will say that $H$ and $K$ are \emph{$\mathcal A_*$-disjoint}.
\end{prop}

\begin{proof}
If \ref{item1} is true, then the subalgebra $\Lie(\mathcal A_* \cap (H \rtimes K))$ of $\Lie(A_*)$ decomposes as a semi-direct product of $\Lie(\mathcal A_* \cap H)$ by $\Lie(\mathcal A_* \cap K)$, thus \ref{item2} holds. 

Conversely, suppose \ref{item1} false. Then there exists $g = hk \in \mathcal A_j \cap (H \rtimes K)$, where neither $h$ nor $k$ belongs to $\mathcal A_j$. This means that $h \equiv k^{-1} \not\equiv 1 \pmod{\mathcal A_j}$. Then there exists $i < j$ such that $h, k \in \mathcal A_i - \mathcal A_{i+1}$, giving a counter-example to our second assertion: $\bar h = - \bar k \neq 0 \in \mathcal A_i/\mathcal A_{i+1}$.
\end{proof}

\subsection{Application to the Andreadakis problem}

We can apply Proposition \ref{dec_of_induced_SCF} to the case when $G = IA_n$ and $\mathcal A_*$ is the Andreadakis filtration. In that case, the Johnson morphism gives an embedding of $\Lie(\mathcal A_*)$ into $\Der(\mathfrak LV)$ (see Example \ref{Johnson_Fn}). Thus, we can check whether $H$ and $K$ are $\mathcal A_*$-disjoint by answering the following question: can an element of $K$ and an element of $H$ induce the same derivation of $\mathfrak LV$ ?

When the subgroups are $\mathcal A_*$-disjoint, then $\mathcal A_* \cap (H \rtimes K)$ is the semi-direct product of $\mathcal A_* \cap H$ by $\mathcal A_* \cap K$. Suppose moreover than the semi-direct product $H \rtimes K$ is an almost-direct one. Then the lower central series $\Gamma_* \cap (H \rtimes K)$ also decomposes as the semi-direct of $\Gamma_* \cap H$ by $\Gamma_* \cap K$. Thus, under these hypotheses, in order to show that $\mathcal A_* \cap (H \rtimes K) = \Gamma_* \cap (H \rtimes K)$, we just need to prove that $\mathcal A_* H = \Gamma_* H$ and that $\mathcal A_* K = \Gamma_* K$. We sum this up in the following:
\begin{theo}\label{Reducing_Andreadakis}
Let $H \rtimes K$ be a subgroup of $IA_n$. Suppose that:
\begin{enumerate}
\item the semi-direct product is an almost-direct one,
\item an element of $H$ and an element of $K$ cannot induce the same derivation of the free Lie algebra through the Johnson morphism,
\item the Andreadakis equality holds for $H$ and $K$,
\end{enumerate}
then the Andreadakis equality holds for $H \rtimes K$.
\end{theo}

\section{First application: triangular automorphisms}\label{Section_triangular}

\begin{defi}\label{def_triangulaires}
Fix $(x_1,..., x_n)$ an ordered basis of $F_n$. The subgroup $IA_n^+$ of $IA_n$ consists of \emph{triangular automorphisms}, \emph{i.e.\ } automorphisms $\varphi$ acting as :
\[\varphi : x_i \longmapsto (x_i^{w_i})\gamma_i,\]
where $w_i \in \langle x_j \rangle_{j < i} \cong F_{i-1}$ et $\gamma_i \in \Gamma_2(F_{i-1})$.
\end{defi}

\subsection{Decomposition as an iterated almost-direct product}

Consider the subgroup of $IA_n^+$ of triangular automorphisms fixing every element of the basis, except for the $i$-th one. This subgroup is the kernel of the projection $IA_i^+ \rightarrow IA_{i-1}^+$ induced by $x_i \mapsto 1$. It is isomorphic to $\Gamma_2(F_{i-1}) \rtimes F_{i-1}$, the isomorphism being:
\[\left(\varphi: x_i \mapsto (x_i^w)\gamma \right) \longmapsto (\gamma, w).\]
Thus we obtain a short exact sequence:
\begin{equation}\label{dec_IA+}
\begin{tikzcd}
\Gamma_2(F_{n-1}) \rtimes F_{n-1} \ar[r, hook] &IA_n^+ \ar[r, two heads] &IA_{n-1}^+.
\end{tikzcd}
\end{equation}
This sequence is split: a section is given by automorphisms fixing $x_n$. Hence, we get a decomposition into a semi-direct product:
\begin{equation}\label{dec_IA+_psd}
IA_n^+ = (\Gamma_2(F_{n-1}) \rtimes F_{n-1})\rtimes IA_{n-1}^+.
\end{equation} 

\begin{lem} 
Let $G$ be a group. For any integer $i$: 
\[\Gamma_*(\Gamma_iG \rtimes G) = \Gamma_{*+i-1}(G) \rtimes \Gamma_*G.\]
\end{lem}
\begin{proof}
From the definition of the lower central series and from Definition \ref{SCD_relative} ($G$ acting on $\Gamma_iG$ by conjugation), we immediately deduce the following equality, true for every $i$:
\[\Gamma_*^G(\Gamma_iG) = \Gamma_{*+i-1}(G).\] 
The result follows, by Proposition \ref{Gamma(PSD)_0}.
\end{proof}

In particular, for $G = F_n$ and $i = 2$, this determines the lower central series of $\Gamma_2(F_{n-1}) \rtimes F_{n-1}$:
\[\Gamma_*(\Gamma_2(F_{n-1}) \rtimes F_{n-1}) = \Gamma_{*+1}(F_{n-1}) \rtimes \Gamma_* F_{n-1}.\]
Whence the following description of its abelianization: 
\[(\Gamma_2(F_{n-1}) \rtimes F_{n-1})^{ab} = \Gamma_2/\Gamma_3(F_{n-1}) \times F_{n-1}^{ab} = (\Lie_2 \oplus \Lie_1)(F_{n-1}).\] 
The extension \eqref{dec_IA+} induces an action of $IA_{n-1}^+$ on this abelianization. One can easily check that this action is none other than the diagonal action induced by the canonical action of $IA_{n-1}^+$ on the Lie algebra of $F_{n-1}$. This action is trivial, by definition of $IA_{n-1}$, which means that the semi-direct product \eqref{dec_IA+_psd} is an almost-direct one. Thus, Proposition \ref{Gamma(PSD)} allows us to describe the lower central series of $IA_n^+$:
\begin{prop}
For every integer $n$:
\begin{align*}
\Gamma_*(IA_n^+) 
&= \Gamma_*(\Gamma_2(F_{n-1}) \rtimes F_{n-1})\rtimes \Gamma_*(IA_{n-1}^+) \\
&= (\Gamma_{*+1}(F_{n-1}) \rtimes \Gamma_*F_{n-1}))\rtimes \Gamma_*(IA_{n-1}^+).
\end{align*}
In particular, the Lie algebra of $IA_n^+$ decomposes as:
\[\Lie(IA_n^+) = (\mathfrak L_{*+1}(\Z^{n-1}) \rtimes \mathfrak L_*(\Z^{n-1}))\rtimes \Lie(IA_{n-1}^+).\]
\end{prop}

\subsection{The Andreadakis equality}

Our aim is to use the decomposition into an almost-direct product described in the previous paragraph to recover the main result of \cite{Satoh_triangulaire}:
\begin{theo}\label{Andreadakis_for_triangulars}
The subgroup of triangular automorphisms satisfies the Andreadakis equality, that is:
\[\Gamma_*(IA_n^+) = \Gamma_*(IA_n)\cap IA_n^+ = \mathcal A_* \cap IA_n^+ = \mathcal A_*(IA_n^+, \Gamma_*(F_n)).\]
\end{theo}

\begin{proof}

Consider the decomposition \eqref{dec_IA+_psd}. We want to apply Theorem \ref{Reducing_Andreadakis} to it. First, we need to show that the factors are $\mathcal A_*$-disjoint. But this is obvious : each $\varphi \in IA_{n-1}^+$ satisfies $[\varphi,x_n] = 1$, whereas each element $\psi$ in the other factor satisfies $[\varphi,x_i] = 1$ for $i < n$. Thus, a derivation $d$ coming from both factors would satisfy $d(x_i) = 0$ for every $i$, so it must be trivial.

We are thus reduced to showing the Andreadakis equality for the first factor, and the result will follow by induction.
Let $\psi: x_n \mapsto x_n^w \cdot\gamma$ be an element of this factor ($w \in F_{n-1}$, $\gamma \in  \Gamma_2 F_{n-1}$, and $\psi$ fixes the other generators). Then:
\begin{align*}
\varphi \in \mathcal A_j 
&\ \Leftrightarrow\ \varphi(x_n) \equiv x_n \pmod{\Gamma_{j+1}(F_n)}\\
&\ \Leftrightarrow\ \gamma \equiv [x_n,w] \pmod{\Gamma_{j+1}(F_n)}.
\end{align*}
We claim that this is possible only if $\gamma \in \Gamma_{j+1}(F_n)$ and $w \in \Gamma_j(F_n)$. Indeed, let $k$ such that $w \in \Gamma_k-\Gamma_{k+1}$ (such a $k$ exists because $F_n$ is residually nilpotent). If we had $k<j$, then: 
\[0 \neq \overline w \in \Lie_k(F_{n-1}) \subseteq \Lie_k(F_n).\]
Since $\Lie_*(F_n)$ is the free Lie algebra over the $\overline x_s$, and $\overline x_n$ does not appear in $\overline w$, then $[\overline x_n,\overline w] \neq 0 \in \Lie_{k+1}(F_n)$, and $[\overline x_n,\overline w]$, containing $\bar x_n$, cannot be in $\Lie(F_{n-1})$. In particular, it cannot be equal to $\overline \gamma$, which contradicts this hypothesis. Thus we must have $k \geq j$, that is $w \in \Gamma_j$, and $\gamma \equiv [x_n, w] \in \Gamma_{j+1}.$

Using the description of the lower central series of $\Gamma_2(F_{n-1}) \rtimes F_{n-1}$ from the previous paragraph, we see that this means exactly that $\psi \in \Gamma_j(\Gamma_2(F_{n-1}) \rtimes F_{n-1})$, whence the Andreadakis equality for this subgroup, which is the desired conclusion.
\end{proof}

We can also state the Andreadakis equality for the triangular McCool subgroup, studied in \cite{CPVW}. It is not a consequence of the previous theorem, but of its proof. It was not obvious from the proof given in \cite{Satoh_triangulaire} ; however, by Lemma \ref{tau_inj}, it is equivalent to the injectivity of the Johnson morphism $\mathcal L (P \Sigma_n^+) \rightarrow \Der(LV)$, that was showed in \cite[Cor. 6.3]{CHP}.

\begin{cor}\label{Andreadakis_for_triangular_McCool}
The subgroup $P \Sigma_n^+$ of triangular basis-conjugating automorphisms satisfies the Andreadakis equality, that is:
\[\Gamma_*(P \Sigma_n^+) = \Gamma_*(IA_n)\cap P \Sigma_n^+ = \mathcal A_* \cap P \Sigma_n^+ = \mathcal A_*(P \Sigma_n^+, \Gamma_*(F_n)).\]
\end{cor}

\begin{proof}
In the proof of Theorem \ref{Andreadakis_for_triangulars}, consider only those factors corresponding to basis-conjugating automorphisms : take all $\gamma$ and $\gamma_i$ to be $1$, and forget the factors $\Gamma_2(F_{n-1})$ corresponding to these elements.
\end{proof}

\section{Second application: the pure braid group}\label{Section_braids}

We refer to \cite{Birman} or the more recent \cite{Birman-Brendle} for a detailed introduction to braid groups. As usual, we denote by $B_n$ Artin's braid group, generated by the $\sigma_i$ ($1 \leq i < n$), and by $P_n$ the subgroup of pure braids, generated by the $A_{ij}$ ($1 \leq i < j \leq n$). 
Recall the geometric description of the generators: 
\medskip

$\begin{array}{|C|C|}
\hline
$\sigma_i$ &$A_{ij} = {}^{(\sigma_{n-1} \cdots \sigma_{i+1})}\!\sigma_i^2$ \\
\hline
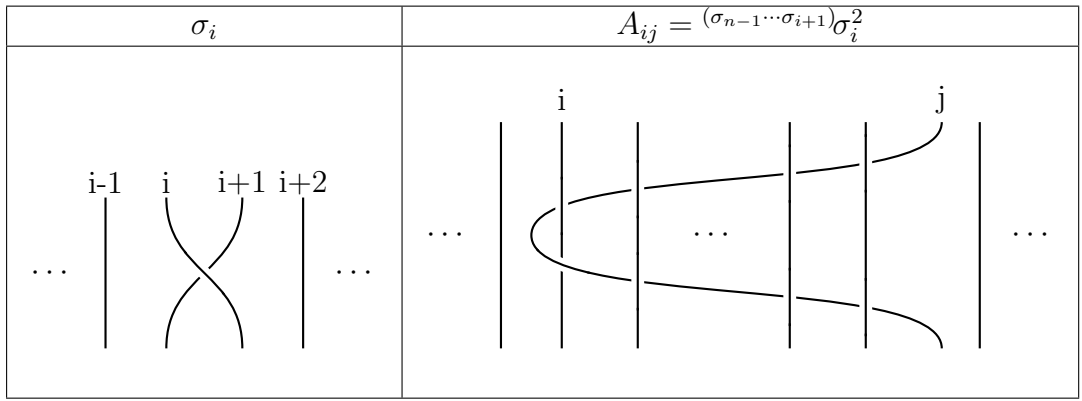
\begin{tikzpicture}
\path (2, 3) (2, -.5);
\node[draw=none] at (-.8,2.2) {i-1};
\node[draw=none] at (0,2.2) {i};
\node[draw=none] at (1,2.2) {i+1};
\node[draw=none] at (1.8,2.2) {i+2};

\node[draw=none] at (-1.5,1) {$\cdots$};
\node[draw=none] at (2.5,1) {$\cdots$};

\begin{knot}[clip width = 6, flip crossing = 1]
\strand[thick]  (-.8, 0)  -- (-.8, 2);
\strand[thick]  (0, 0) .. controls  +(0, 1) and +(0, -1) .. (1, 2);
\strand[thick] (1 ,0) .. controls  +(0, 1) and +(0, -1) .. (0, 2);
\strand[thick]  (1.8, 0)  -- (1.8, 2);
\end{knot}
\end{tikzpicture}
&
\begin{tikzpicture}
\path (2, 4) (2, -.5);
\node[draw=none] at (0,3.3) {i};
\node[draw=none] at (5,3.3) {j};

\node[draw=none] at (-1.5,1.5) {$\cdots$};
\node[draw=none] at (2,1.5) {$\cdots$};
\node[draw=none] at (6.2,1.5) {$\cdots$};

\begin{knot}[clip width = 6, flip crossing = 1]
\strand[thick]  (-.8, 0)  -- (-.8, 3);
\strand[thick]  (0, 0)  -- (0, 3);
\strand[thick]  (1, 0)  -- (1, 3);
\strand[thick]  (3, 0)  -- (3, 3);
\strand[thick]  (4, 0)  -- (4, 3);
\strand[thick]  (5, 0)  .. controls  +(0, 1) and +(0, -1) .. (-.4, 1.5) .. controls  +(0, 1) and +(0, -1) .. (5, 3);
\strand[thick]  (5.5, 0)  -- (5.5, 3);
\end{knot}
\end{tikzpicture}  \\ \hline
\end{array}
$

\medskip

An embedding of $B_n$ into $B_{n+1}$ is given by sending  $\sigma_i$ to $\sigma_i$ (it identifies $B_n$ to the subgroup of braids on the first $n$ strings). Inside $B_{n+1}$, the $A_{i,n+1} =: x_i$ generate a free group $F_n$, which is stable under conjugation by elements of $B_n$. This conjugation action is called the \emph{Artin action} ; explicitely, $\sigma_i$ acts \emph{via} the automorphism:
\begin{equation}\label{Artin_action}
\left\{\begin{array}{lll}
                 x_i     &\longmapsto & {}^{x_i}\! x_{i+1}. \\
				 x_{i+1} &\longmapsto & x_i   
		 \end{array} \right.
\end{equation}
That $\langle A_{i,n+1} \rangle$ is a free group can be seen using a geometric argument: this subgroup is the kernel of the projection of $P_{n+1}$ onto $P_n$ obtained by forgetting the $n$-th string, and this kernel identifies canonically with $\pi_1(\mathbb R^2 - \{n \text{ points}\})$ \cite[th. 1.4]{Birman}.
The above surjection of $P_{n+1}$ onto $P_n$ is split, a splitting been given by the above inclusion of $B_n$ into $B_{n+1}$. We thus get a decomposition of the pure braid group as a semi-direct product:
\begin{equation}\label{dec_Pn}
P_{n+1} = P_n \ltimes F_n.
\end{equation}

This decomposition allows us to write any $\beta$ in $P_n$ uniquely as $\beta'\beta_n$, with $\beta' \in P_n$ and $\beta_n \in \langle A_{1, n}, ..., A_{n-1, n}\rangle \cong F_{n-1}$. Iterating this, we obtain a unique decomposition of $\beta$ as:
\[\beta = \beta_1 \cdots \beta_n,\ \text{ avec }\ \beta_k \in \langle A_{1, k}, ..., A_{k-1, k}\rangle \cong F_{k-1}.\]
We then say that we have \emph{combed} the braid $\beta$.
This is key in the proof of the following:
\begin{prop}\label{Artin_fidèle}
The Artin action of $B_n$ on $F_n$ is faithful. 
\end{prop}

We can thus embed $B_n$ into $\Aut(F_n)$. We will often identify $B_n$ with its image in $\Aut(F_n)$, even if this embedding depends on the choice of an ordered basis of $F_n$. The corresponding action of $B_n$ on $F_n^{ab}$ is by permutation of the corresponding basis. Thus, for any choice of basis:
\begin{equation}\label{braids_and_IA}
B_n \cap IA_n = P_n.
\end{equation}

\subsection{Decomposition as an iterated almost-direct product}\label{SCD_de_Pn}

Consider the decomposition \eqref{dec_Pn} of the pure braid group. Since $P_n$ acts through $IA_n$ on $F_n$ \eqref{braids_and_IA}, this decomposition is an almost-direct-product. From Proposition \ref{Gamma(PSD)}, we deduce that the lower central series also decomposes:
\[\forall j,\ \Gamma_j(P_n) = \Gamma_j(P_{n-1}) \ltimes \Gamma_j(F_{n-1}),\]
and so does the Lie ring of $P_n$:
\begin{equation}\label{dec_Lie(P_n)}
\Lie(P_n) = \Lie(P_{n-1}) \ltimes \Lie(F_{n-1}) = \Lie(P_{n-1}) \ltimes \mathfrak L(\mathbb Z^{n-1}).
\end{equation}
Thus, $\Lie(P_n)$ decomposes as an iterated semi-direct product of free Lie algebras. With a little more work, using the classical presentation of $P_n$, we can get a presentation of this Lie ring, which is none other than the Drinfeld-Kohno Lie ring (Proposition \ref{Drinfeld-Kohno} in the appendix). We refer the reader to the original work \cite{Kohno}, or to the recent book \cite[section 10.0]{Benoit} for more on this algebra, with a somewhat different point of view.

\subsection{The Andreadakis equality}

\begin{theo}\label{Andreadakis_for_braids}
The pure braid group, embedded into $IA_n$ \emph{via} the Artin action, satisfies the Andreadakis equality:
\[\Gamma_*(P_n) = \Gamma_*(IA_n)\cap P_n = \mathcal A_* \cap P_n = \mathcal A_*(P_n, \Gamma_*(F_n)).\]
\end{theo}

\begin{proof}
We apply theorem \ref{Reducing_Andreadakis} to the decomposition $P_n = P_{n-1} \ltimes F_{n-1}$ described above. $\mathcal A_*$-disjointness is easy to verify: any $\beta \in P_{n-1}$ commutes with braids on the strings $n$ and $n+1$, so $[\beta, x_n]=1$ (where $x_n = A_{n,n+1}$), whereas no $w$ in $F_{n-1} = \langle A_{1, n}, ..., A_{n-1, n}\rangle$ can commute with $x_n$, because $\langle A_{1, n}, ..., A_{n-1, n}, A_{n,n+1}\rangle$ is a free group (it is the one obtained by exchanging the roles of the strings $n$ and $n+1$ in the arguments above). Thus no $[\bar\beta,-]$ can coincide with some non-trivial $[\bar w,-]$.

In order to show that the factor $F_{n-1}$ satisfies the Andreadakis equality (from which the result follows by induction, using Theorem \ref{Reducing_Andreadakis}), we need the following lemma :
\begin{lem}\label{word_lemma}
Let $w \in F_n$ such that for some $i$, $[w, x_i] \in \Gamma_{j+1}(F_n)$. Then: 
\[\exists n \in \mathbb Z,\ wx_i^n \in \Gamma_j(F_n).\]
In particular, if $w \in \langle x_1, ..., \hat x_i, ..., x_n \rangle$, or if $w \in \Gamma_2(F_n)$, then $n$ must be $0$, so that $w \in \Gamma_j(F_n)$.
\end{lem}

Let $w \in \mathcal A_j(F_n) \cap \langle A_{1, n}, ..., A_{n-1, n} \rangle = F_{n-1}$. Then $[w, x_n]$ is in $\Gamma_{j+1}(F_n)$, which is contained in  $\Gamma_{j+1}(P_{n+1})$, by the above calculation of the lower central series of $P_{n+1}$. 
Observe that $w$ and $x_n$ belong to $\langle A_{1, n}, ..., A_{n-1, n}, A_{n,n+1}\rangle$, which is another copy of the free group on $n$ generators in the braid group, that will be called $\tilde F_n$. Exchanging the roles of the strings $n$ and $n+1$ in the previous paragraph give an almost-direct product decomposition $P_{n+1} = P_n \ltimes \tilde F_n$ and, using Proposition \ref{SCD_de_Pn}:
\[\Gamma_{j+1}(P_{n+1}) \cap \tilde F_n = \Gamma_{j+1}(\tilde F_n).\]
Thus $[w, x_n]\in \Gamma_{j+1}(\tilde F_n)$. But since $w \in \langle A_{1, n}, ..., A_{n-1, n} \rangle = F_{n-1}$, the generator $x_n = A_{n, n+1}$ of $\tilde F_n$ does not appear in $w$. We deduce the conclusion we were looking for, using Lemma \ref{word_lemma}:
$w \in \Gamma_j(\tilde F_n) \cap F_{n-1} = \Gamma_j(F_{n-1}) \subseteq \Gamma_j(P_n).$
\end{proof}

\begin{proof}[Proof of Lemma \ref{word_lemma}]
Let $k$ such that $w \in \Gamma_k -\Gamma_{k+1}$ (such a $k$ exists since $F_n$ is residually nilpotent). If we had $2 \leq k < j$, then $[\bar w, x_i] = 0$ inside $\Lie_{k+1}(F_n) = \mathfrak L_{k+1}(V)$, since $k+1 < j+1$. Hence $\bar w$ would be a non-trivial element of the centralizer $C(x_i)$ of $x_i$ in the free Lie algebra $\mathfrak L(V)$. But $C(x_i) = \mathbb Z x_i \subseteq \mathfrak L_1(V)$, so $k$ must be $1$, which contradicts our hypothesis.

If $w \in F_n -\Gamma_2$, then $\bar w \in C(x_i) = \mathbb Z x_i$, so there is an $n$ such that $w \equiv x_i^n \pmod{\Gamma_2}$. Then $wx_i^{-n}$ is in $\Gamma_2$ and it satisfies the hypothesis of the lemma, so it is in $\Gamma_j$, using the first part of the proof.
\end{proof}

\begin{rmq}[G. Massuyeau]
The Andreadakis filtration on braids is none other than the one given by vanishing of the first Milnor invariants. Precisely, a pure braid $\beta$ acts by conjugation on the (fixed) basis of the free group, $x_i$ being acted upon by the \emph{parallel} $w_i$ \cite{Milnor, Ohkawa}. The braid is in $\mathcal A_j$ if and only if the parallel is in $\Gamma_j F_n$ (use the above proof), that is, if and only if its Magnus expansion is in $1 + (X_1, ... X_n)^k$. This last condition is exactly the vanishing of its Milnor invariants of length less than $k$. Thus, our Theorem \ref{Andreadakis_for_braids} can also be interpreted as a consequence of two facts well-known to knot theorists: Milnor invariants of length at most $d+1$ of pure braids generate Vassiliev invariants of degree at most $d$ ; a braid is in $\Gamma_{d+1}P_n$ if and only if it is undetected by Vassiliev invariants of degree at most $d$ \cite{Habegger, Mostovoy}.
\end{rmq}

\begin{subappendices}

\section{Appendix: Generating sets and relations}\label{appendix_gen}

A finite set of generators of $IA_n$ has been known for a long time \cite{Nielsen} -- see also \cite[5.6]{BBM}. These are:
\begin{equation}\label{gen_of_IAn}
K_{ij}: x_t \longmapsto \begin{cases}
                                 x_j x_i x_j^{-1} &\text{if } t = i \\
						                     x_t              &\text{else }    
													 \end{cases}
\ \ \text{ and }\ \ 
K_{ijk }: x_t \longmapsto \begin{cases}
                                [x_j, x_k] x_i &\text{if } t = i \\
						                     x_t              &\text{else. }    
														\end{cases}
\end{equation}

\subsection{Generators of \texorpdfstring{$IA_n^+$}{IAn}}

We give a family of generators of the group of triangular automorphisms:
\begin{lem}\label{gen_de_IA+}
$IA_n^+$ is generated by the following elements:
\[\left\{\begin{array}{llll}
         K_{ij}: x_i &\longmapsto  & x_i^{x_j} & \text{ for } j < i, \\
		 K_{ijk}: x_i &\longmapsto & x_i[x_j,x_k]   & \text{ for } j,k < i, 
		 \end{array} \right.\]
\end{lem}

\begin{proof}
Let $G$ be the subgroup generated by the above elements.
Let $\varphi \in IA_n^+$, sending each $x_i$ on some $(x_i^{w_i})\gamma_i$ (as in Definition \ref{def_triangulaires}). We can decompose $\varphi$ as $\varphi = \varphi_n \circ \cdots \circ \varphi_2$, where $\varphi_i$ fixes all elements of the basis, except for the $i$-th one, that is sends to $(x_i^{w_i})\gamma_i$. We claim that each $\varphi_i$ is in $G$.
Let us fix $i$, and consider only automorphisms fixing all elements of the basis, save the $i$-th one. One can check that the subgroup of $G$ generated by the $K_{ij}$ ($j<i$) is the set of automorphism of the form $c_{i, w}: x_i \longmapsto x_i^w$ (for $w \in F_{i-1}$).Using the formulas $[a,bc] =[a,b]\cdot({}^b\![a,c])$, $[a,b]^{-1} = [b,a]$ and $[a^{-1}, b] = [b,a]^a$, we can decompose the product of commutators $\gamma_i$ as:
\[\gamma_i = \prod\limits_{k = 1}^m [x_{\alpha_k}, x_{\beta_k}]^{\omega_k},\]
where $\omega_k \in F_{i-1}$.
Thus
$\varphi_i = c_{i, w_i} \circ K_{i,\alpha_1, \beta_1}^{c_{i, \omega_1}} \circ \cdots \circ K_{i,\alpha_m, \beta_m}^{c_{i, \omega_m}},$
whence the desired conclusion.
\end{proof}

\begin{rmq}
Fix $i \geq 1$. From the above proof, we see that the automorphisms:
\[\left\{\begin{array}{lllll}
                 x_i &\longmapsto & x_i^{x_j} & \text{ pour } j < i  
                 & (x_j \in F_{i-1}),\\
				 x_i &\longmapsto & x_i[x_j,x_k]   & \text{ pour } j,k < i
				 & ([x_j,x_k] \in \Gamma_2(F_{i-1})),
				        \end{array} \right.\]
generate the subgroup of $IA_n^+$ of triangular automorphisms fixing every element of the basis, save the $i$-th one.
\end{rmq}

\subsection{Presentation of the Drinfeld-Kohno Lie ring}

The Drinfeld-Kohno Lie ring is the Lie ring $\Lie(P_n)$. Our methods can be used to recover the usual presentation of this Lie ring:

\begin{prop}\label{Drinfeld-Kohno}
The Lie ring of $P_n$ is generated by $t_{ij}$ ($1 \leq i , j \leq n$), under the relations:
\[\begin{cases}
t_{ij} = t_{ji},\ t_{ii} = 0 &\forall i,j,\\ 
[t_{ij}, t_{ik} + t_{kj}] = 0 &\forall i, j, k,\\
[[t_{ij}, t_{kl}] = 0 &\text{si } \{i,j\} \cap \{k,l\} = \varnothing.
\end{cases}\]
\end{prop}
\begin{proof}
We use the classical presentation of the pure braid group (see for instance \cite{Birman-Brendle}) given, for $r<s<n$ and $i<n$, denoting $A_{\alpha, n}$ by $A_{\alpha, n}$:
\[[A_{rs},x_i] = [A_{rs},A_{in}] = 
\left\{
\begin{array}{l}
1 \text{ si } s<i \text{ ou } i<r,\\ 

[x_i^{-1},x_r^{-1}] \text{ si } s=i, \\ 

[x_s^{-1}, x_i] \text{ si } r=i,\\

[[x_r,x_s]^{-1},x_i] \text{ si } r<i<s.
\end{array}
\right.\]
These relations are easily verified ; that they are enough to describe pure braids comes from the decomposition into free factors. We use a similar reasoning to show our proposition.

Let $\mathfrak p_n$ be the Lie ring defined by the presentation of the proposition. The relations of the above presentations of $P_n$ imply that the classes of $A_{ij}$ in $\Lie(P_n)$ (which generate it by \ref{engdeg1}) satisfy these relations, whence a morphism $u_n$ from $\mathfrak p_n$ onto $\Lie(P_n)$. We also can define a split epimorphism $\pi_n$ from $\mathfrak p_n$ onto $\mathfrak p_{n-1}$ by sending $t_{ij}$ on $0$ if $n \in \{i,j\}$, and on $t_{ij}$ else. Denote by $\mathfrak k_n$ the kernel of $\pi_n$, and consider the diagram:
\[\begin{tikzcd}
\mathfrak k_n \ar[r, hook] \ar[d, "v_n"] &\mathfrak p_n \ar[r, two heads, "\pi_n"] \ar[d, "u_n"] &\mathfrak p_{n-1} \ar[d, "u_{n-1}"] \\
\Lie(F_{n-1}) \ar[r, hook] &\Lie(P_n) \ar[r, two heads] &\Lie(P_{n-1})
\end{tikzcd}\]
where the second line comes from the decomposition \eqref{dec_Lie(P_n)}.
The algebra $\mathfrak k_n$ is the ideal of $\mathfrak p_n$ generated by the $t_{in}$. Since the subalgebra generated by the $t_{in}$ already is an ideal, $\mathfrak k_n$ is generated by the $t_{in}$ as a Lie algebra. Since $v_n$ sends the $t_{in}$ on a basis of the free Lie algebra $\Lie(F_{n-1})$, it is an isomorphism (the universal property of the free Lie algebra gives an inverse). Remark that $u_1$ obviously is an isomorphism. By induction, using the five lemma, we deduce that $u_n$ is an isomorphism.
\end{proof}

\end{subappendices}

\bibliographystyle{alpha}
\bibliography{Ref_SS-groupes}

\end{document}